\documentclass[onecolumn,english]{svjour3}
\pdfoutput=1
\usepackage{geometry}
\usepackage[T1]{fontenc}
\usepackage[latin9]{inputenc}
\usepackage{textcomp}
\usepackage{refstyle}

\usepackage{amsmath}
\usepackage{amssymb}
\usepackage{graphicx}
\usepackage{refstyle}
\usepackage{esint}

\makeatletter
\smartqed  
\numberwithin{equation}{section}
\numberwithin{theorem}{section}
\numberwithin{proposition}{section}
\makeatother

\usepackage{babel}
\begin{document}

\title{Extremal Trajectories and Maxwell Strata in Sub-Riemannian Problem on Group of Motions of Pseudo Euclidean Plane}

\author{Yasir Awais Butt, Yuri L. Sachkov \thanks{Work of the second author is supported by 
Grant of the Russian Federation for the State Support of Researches
(Agreement  No~14.B25.31.0029).}, Aamer Iqbal Bhatti}
\institute{Yasir Awais Butt \at Department of Electronic Engineering\\
Muhammad Ali Jinnah University\\
Islamabad, Pakistan\\
Tel.: +92-51-111878787\\
\email{yasir\_awais2000@yahoo.com}\\
\and Yuri L. Sachkov  \at Program Systems Institute\\
Pereslavl-Zalessky, Russia\\
\email{sachkov@sys.botik.ru}\\
\and Aamer Iqbal Bhatti\at Department of Electronic Engineering\\
Muhammad Ali Jinnah University\\
Islamabad, Pakistan\\
Tel.: +92-51-111878787\\
\email{aib@jinnah.edu.pk}}

\date{Received: date / Accepted: date}
\titlerunning{Sub-Riemannian Problem on Group SH(2)}
\maketitle

\begin{abstract}
We consider the sub-Riemannian length minimization problem on the group of motions of pseudo Euclidean plane that form the special hyperbolic group $\mathrm{SH}(2)$. The system comprises of left invariant vector fields with 2-dimensional linear control input and energy cost functional. We apply the Pontryagin Maximum Principle to obtain the extremal control input and the sub-Riemannian geodesics. A change of coordinates transforms the vertical subsystem of the normal Hamiltonian system into the mathematical pendulum. In suitable elliptic coordinates the vertical and the horizontal subsystems are integrated such that the resulting extremal trajectories are parametrized by the Jacobi elliptic functions. Qualitative analysis reveals that the projections of normal extremal trajectories on the $xy$-plane have cusps and inflection points. The vertical subsystem being a generalized pendulum admits reflection symmetries that are used to obtain a characterization of the Maxwell strata.


\end{abstract}

\section{Introduction\label{sec:Introduction}}

Sub-Riemannian geometry deals with the study of smooth manifolds $M$ that are endowed with a vector distribution $\Delta$ and a smoothly varying positive definite quadratic form. The distribution $\Delta$ is a subbundle of the tangent bundle $TM$ and the quadratic form allows measuring distance between any two points in $M$ \cite{Strichartz},\cite{Montgomery_Book},\cite{Agrchev_Barilari_Boscain_SR}. Other names that appear in literature for sub-Riemannian Geometry are Carnot-Carathéodory geometry \cite{Gromov}, Non-holonomic Riemannian geometry \cite{Gershkovich} and Singular Riemannian geometry \cite{R_W_Brocket}. The aim of defining and solving the sub-Riemannian problem is to find
the optimal curves between two given points on the sub-Riemannian manifold $M$ such that the sub-Riemannian length between the points is minimized \cite{Montgomery_Book},\cite{Agrchev_Barilari_Boscain_SR}. Sub-Riemannian problems occur widely in nature and technology \cite{Montgomery_Book},\cite{SR-Examples} and have therefore been extensively studied via the geometric control methods on various Lie groups such as the Heisenberg group \cite{Gershkovich},\cite{R_W_Brocket},\cite{Heisenberg_Group}, $\mathrm{S^{3}}$, SL(2), SU(2) \cite{SO(3)}, SE(2) \cite{max_sre},\cite{cut_sre1},\cite{cut_sre2}, the Engel group \cite{Ardentov}, the Solvable group $\mathrm{SOLV}^{-}$ \cite{Mazhitova}, and also in \cite{Sachkov_Dido_Symmetries},\cite{Sachkov-Dido-Max}. Among the physical systems that describe sub-Riemannian problems and on which Geometric control methods have been successfully applied one can find parking of cars, rolling bodies on a plane without sliding, motion planning and control of robots, satellites, vision, quantum mechanical systems and even finance \cite{Montgomery_Book},\cite{SR-Examples}. 

We consider the sub-Riemannian problem on the group of motions of the Pseudo-Euclidean plane which is a subspace of the pseudo Euclidean space \cite{Wong}. The motions of pseudo Euclidean plane described in Section \ref{subsec:3.1} form a 3-dimensional Lie group known as the special hyperbolic group SH(2) \cite{Ja.Vilenkin}. The optimal control problem comprises a system
of left invariant vector fields with 2-dimensional linear control input and energy cost functional. The group SH(2) gives one of the Thurston's 3-dimensional geometries \cite{Thurston} and the study of sub-Riemannian problem on SH(2) bears significance in the program of complete study of all the left-invariant sub-Riemannian problems on 3-dimensional Lie groups following the classification in terms of the basic differential invariants \cite{agrachev_barilari}. 

Notice that equivalent sub-Riemannian problem was considered in \cite{Mazhitova} on the Lie group $\mathrm{SOLV}^{-}$. However, the parametrization of sub-Riemannian geodesics obtained in \cite{Mazhitova} is far from being complete. This paper seeks the rigorous scheme of analysis developed in \cite{max_sre},\cite{Ardentov},\cite{Sachkov_Dido_Symmetries} for parametrization and qualitative analysis of the extremal trajectories. The corresponding results are insightful, in simpler form owing primarily
to the use of simpler elliptic coordinates and allow further analysis on global and local optimality of geodesics. The paper is organized as follows. We begin with a short overview of sub-Riemannian geometry in Section 2. We present detailed account of problem statement in Section 3 covering  a description of the Lie group $\mathrm{SH(2)}$ and the sub-Riemannian problem on $\mathrm{SH(2)}$. In Section 4 we apply the Pontryagin Maximum Principle to obtain the normal Hamiltonian system. In Section 5 we present the computation of the Hamiltonian flow in the elliptic coordinates and the qualitative analysis of the projections of the extremal trajectories on the $xy$-plane. Section 6 pertains to the reflection symmetries of the vertical and horizontal subsystem. In Section 7 we utilize results from reflection symmetries of the Hamiltonian system to state the generalized conditions for the Maxwell strata. Sections 8 and 9 pertain to the future work and the conclusion respectively.

\section{Sub-Riemannian Geometry}
\subsection{Sub-Riemannian Manifold}

A Sub-Riemannian space/manifold is a generalization of a Riemannian manifold. It comprises of a manifold $M$ of dimension $n$, a smooth vector distribution $\Delta$ with constant rank $m\leq n$, and a Riemannian metric $g$ defined on $\Delta$. It is denoted as a triple $(M,\Delta,g)$. The distribution $\Delta$ on $M$ is a smooth linear subbundle of the tangent bundle $TM$. Intuitively, the motion on a sub-Riemannian manifold is restricted along the paths that are tangent to the horizontal subspaces or the admissible directions of motion are constrained to the horizontal subspaces $\Delta_{q}$, $q\in M$ \cite{agrachev_sachkov}, \cite{sachkov_lectures}. 
\subsection{Sub-Riemannian Distance}
A horizontal curve in a sub-Riemannian manifold $(M,\Delta,g)$ is a Lipschitzian curve $\gamma:I\subset\mathbb{R}\rightarrow M$ such that $\dot{\gamma}(t)\in\Delta_{\gamma(t)}$
for almost all $t\in I$. The length of $\gamma$ is given as:
\[
length(\gamma)=\intop_{I}\sqrt{g_{\gamma(t)}(\dot{\gamma}(t),\dot{\gamma}(t))}dt,
\]
where $g_{\gamma(t)}$ is the inner product in $\Delta_{\gamma(t)}$ \cite{Agrchev_Barilari_Boscain_SR}. The sub-Riemannian distance between two points $p,\, q\in M$ is the length of the shortest curve joining $p$ to $q$:
\[
d(p;q)=inf\left\{ length(\gamma):\begin{array}{c}
\gamma\ is\ horizontal\ curve\\
\gamma\ joins\ p\ to\ q
\end{array}\right\} .
\]
\subsection{Sub-Riemannian Problem}

The problem of finding horizontal curves $\gamma$ from the initial state $q_{0}$ to the final state $q_{1}$ with the shortest sub-Riemannian length is called the sub-Riemannian problem \cite{Agrchev_Barilari_Boscain_SR}, \cite{Montgomery_Sub-Riemannian}.
Suppose that there exists a set of smooth vector fields $f_{1},\ldots,f_{m}$ on $M$ whose values $\forall q\in M$ form an orthonormal frame of the Euclidean space $(\Delta_{q},g_{q})$. Sub-Riemannian minimizers are the solutions of the following optimal control problem on $M$:
\begin{eqnarray*}
\dot{q} & = & \sum_{i=1}^{m}u_{i}(t)f_{i}(q),\qquad q\in M,\qquad(u_{1},\cdots,u_{m})\in\mathbb{R}^{m},\\
q(0) & = & q_{0},\qquad q(t_{1})=q_{1},\\
l & = & \int_{0}^{t_{1}}\sqrt{\sum_{i=1}^{m}u_{i}^{2}}dt\to\min.
\end{eqnarray*}
If $(\Delta,g)$ is a left-invariant sub-Riemannian structure on a Lie group, then it has a global orthonormal frame of left-invariant vector fields. Moreover, one can take the initial point $q_{0}=Id$, the identity element of the Lie group.
\section{Problem Statement}
\subsection{The Group SH(2) of Motions of Pseudo Euclidean Plane\label{subsec:3.1}}

The following presentation is motivated from \cite{Ja.Vilenkin} and is presented here for the sake of completeness. The group $\mathrm{SH(2)}$ can be represented by $3 \times 3$ matrices:
\[
M=\mathrm{SH}(2)=\left\{ \left(\begin{array}{ccc}
\cosh z & \sinh z & x\\
\sinh z & \cosh z & y\\
0 & 0 & 1
\end{array}\right)\mid x,y,z\in\mathbb{R}\right\} .
\]
The Lie group $\mathrm{SH(2)}$ comprises of three basis one-parameter subgroups given as:
\[
w_{1}(t)=\left(\begin{array}{ccc}
\cosh t & \sinh t & 0\\
\sinh t & \cosh t & 0\\
0 & 0 & 1
\end{array}\right),\quad w_{2}(t)=\left(\begin{array}{ccc}
1 & 0 & t\\
0 & 1 & 0\\
0 & 0 & 1
\end{array}\right),\quad w_{3}(t)=\left(\begin{array}{ccc}
1 & 0 & 0\\
0 & 1 & t\\
0 & 0 & 1
\end{array}\right),
\]
where $t\in\mathbb{R}$. The basis for the Lie algebra $\mathrm{sh(2)}=T_{Id}\mathrm{SH(2)}$ are the tangent matrices $A_{i}=\frac{dw_{i}(t)}{dt}\mid_{t=0}$ to the subgroups of the Lie group $\mathrm{SH(2)}$. $A_{i}$ are given as:
\[
A_{1}=\left(\begin{array}{ccc}
0 & 1 & 0\\
1 & 0 & 0\\
0 & 0 & 0
\end{array}\right),\quad A_{2}=\left(\begin{array}{ccc}
0 & 0 & 1\\
0 & 0 & 0\\
0 & 0 & 0
\end{array}\right),\quad A_{3}=\left(\begin{array}{ccc}
0 & 0 & 0\\
0 & 0 & 1\\
0 & 0 & 0
\end{array}\right).
\]
The Lie algebra is thus:
\[
\mathcal{L}=T_{Id}M=\mathrm{sh}(2)=\mathrm{span}\left\{ A_{1},A_{2},A_{3}\right\} .
\]
The multiplication rule for $\mathcal{L}$ is $[A,B]=AB-BA$. Therefore, the Lie bracket for $\mathrm{sh}(2)$ is given as $[A_{1},A_{2}]=A_{3}$,
$[A_{1},A_{3}]=A_{2}$ and $[A_{2},A_{3}]=0$. It is trivial to see that the Lie group $\mathrm{SH(2)}$ represents isometries of pseudo Euclidean plane. 

\subsection{Sub-Riemannian Problem on SH(2)}

Consider the following sub-Riemannian problem on $\mathrm{SH(2)}$:
\begin{eqnarray}
\dot{q} & = & u_{1}f_{1}(q)+u_{2}f_{2}(q),\qquad q\in M=\mathrm{SH(2)},\qquad  (u_{1},u_{2})\in\mathbb{R}^{2},\label{eq:3-1}\\
q(0) & = & Id,\qquad q(t_{1})=q_{1},\label{eq:3-2}\\
l & = & \int_{0}^{t_{1}}\sqrt{u_{1}^{2}+u_{2}^{2}}\, dt\to\min,\label{eq:3-3}\\
f_{1}(q) & = & qA_{2},\qquad f_{2}(q)=qA_{1}.\label{eq:3-4}
\end{eqnarray}
In terms of the classification of the left-invariant sub-Riemannian structures on 3D Lie groups \cite{agrachev_barilari}, the canonical frame on $M$ is given as:
\[
f_{1}(q),\ f_{2}(q),\ f_{0}(q)=qA_{3},
\]
\begin{equation*}
[f_{1},f_{0}]=0,\qquad[f_{2},f_{0}]=f_{1},\qquad[f_{2},f_{1}]=f_{0}.
\end{equation*}
By \cite{agrachev_barilari}, the sub-Riemannian structure:
\[
(M,\Delta,g),\qquad\Delta=span\{f_{1},f_{2}\},\qquad g(f_{i},f_{j})=\delta_{ij},
\]
is unique up to rescaling, left invariant contact sub-Riemannian structure on $\mathrm{SH(2)}$. Here $\delta_{ij}$ is the Kronecker delta. 
In the coordinates $(x,y,z)$, the basis vector field are given as: 
\[
f_{1}(q)=\cosh z\frac{\partial}{\partial x}+\sinh z\frac{\partial}{\partial y},
\]
and
\[
f_{2}(q)=\frac{\partial}{\partial z}.
\]
Therefore (\ref{eq:3-1}) may be written as: 
\begin{equation}
\left(\begin{array}{c}
\dot{x}\\
\dot{y}\\
\dot{z}
\end{array}\right)=\left(\begin{array}{c}
\cosh z\\
\sinh z\\
0
\end{array}\right)u_{1}+\left(\begin{array}{c}
0\\
0\\
1
\end{array}\right)u_{2}.\label{eq:3-5}
\end{equation}
By the Cauchy-Schwarz inequality,
\[
(l(u))^{2}=\left(\intop_{0}^{t_{1}}\sqrt{u_{1}^{2}+u_{2}^{2}}dt\right)^{2}\leq t_{1}\intop_{0}^{t_{1}}(u_{1}^{2}+u_{2}^{2})dt,
\]
moreover, the inequality turns into an equality iff $u_{1}^{2}+u_{2}^{2}\equiv \textrm{constant}$. Thus the sub-Riemannian length functional minimization problem (\ref{eq:3-3}) is equivalent to the problem of minimizing the following energy functional with fixed $t_{1}$ \cite{sachkov_lectures}:
\begin{equation}
J=\frac{1}{2}\intop_{0}^{t_{1}}(u_{1}^{2}+u_{2}^{2})dt\rightarrow \min.\label{eq:3-6}
\end{equation}

\subsection{Controllability and Existence of Solutions}
System (\ref{eq:3-1}) has full rank because $f_{0}(q)$=$[f_{2}(q),f_{1}(q)]$=$qA_{3}$. The Lie algebra of the distribution $\mathcal{L}_{q}\Delta$ is given as:
\[
\mathcal{L}_{q}\Delta=span\{f_{1}(q),f_{2}(q),-f_{0}(q)\}=T_{q}\mathrm{SH(2)}\qquad\forall q\in \mathrm{SH(2)}.
\]
Hence by Rashevsky-Chow's theorem, system (\ref{eq:3-1}) is completely controllable \cite{Ravchevsky},\cite{Chow}. Existence of optimal trajectories for the optimal control problem (\ref{eq:3-1})-(\ref{eq:3-4}) follows from Filippov's theorem \cite{max_sre},\cite{agrachev_sachkov}.

\section{Pontryagin Maximum Principle for the Sub-Riemannian Problem on SH(2)}

We write the PMP form for (\ref{eq:3-1}),(\ref{eq:3-2}),(\ref{eq:3-6}) using the coordinate free approach described in \cite{agrachev_sachkov}. Consider the control dependent Hamiltonian for PMP corresponding to the vector fields $f_{1}(q)$ and $f_{2}(q)$:
\begin{equation}
h_{u}^{\nu}(\lambda)=\langle\lambda,u_1 f_{1}(q)+u_2 f_{2}(q)\rangle+\frac{\nu}{2}(u_{1}^{2}+u_{2}^{2}),\quad q=\pi(\lambda),\quad\lambda\in T^{*}M.\label{eq:4-1}
\end{equation}
Let $h_{i}(\lambda)=\langle\lambda,f_{i}(q)\rangle$ be the Hamiltonians corresponding to the basis vector fields $f_{i}$. Then (\ref{eq:4-1}) can be written as:
\begin{equation}
h_{u}^{\nu}(\lambda)=u_{1}h_{1}(\lambda)+u_{2}h_{2}(\lambda)+\frac{\nu}{2}(u_{1}^{2}+u_{2}^{2}),\quad u\in\mathbb{R}^{2},\quad \nu \in \mathbb{R},\quad \lambda \in T^{*}M.\label{eq:4-2}
\end{equation}
Now PMP for the optimal control problem is given by using Theorem 12.3 \cite{agrachev_sachkov} as:
\begin{theorem}
Let $\tilde{u}(t)$ be an optimal control and let $\tilde{q}(t)$ be the associated optimal trajectory for $t\in[0,t_{1}]$. Then, there exists a nontrivial pair:
\[
(\nu,\lambda_{t})\neq0,\qquad\nu\in\mathbb{R},\quad\lambda_{t}\in T_{\tilde{q}(t)}^{*}M,\quad\pi(\lambda_{t})=\tilde{q}(t),
\]
where $\nu\in\{-1,0\}$ is a number and $\lambda_{t}$ is a Lipschitzian curve for which the following conditions hold for almost all times $t\in[0,t_{1}]$:
\begin{eqnarray}
\dot{\lambda}_{t} & = & \overrightarrow{h}{}_{\tilde{u}(t)}^{\nu}(\lambda_{t}),\label{eq:4-3}\\
h_{\tilde{u}(t)}^{\nu}(\lambda_{t}) & = & \underset{u\in\mathbb{R}^{2}}{\max}h_{u}^{\nu}(\lambda_{t}),\label{eq:4-4}
\end{eqnarray}
where $\overrightarrow{h}{}_{\tilde{u}(t)}^{\nu}(\lambda_{t})$ is
the Hamiltonian vector field on $T^{*}M$ corresponding to the Hamiltonian function $h_{\tilde{u}(t)}^{\nu}$.
\end{theorem}
\subsection{Extremal Trajectories}

Extremal trajectories comprise of the abnormal extremal trajectories for $\nu=0$ and normal extremal trajectories for $\nu=-1$. Since the structure is 3-dimensional contact, hence there are no nontrivial abnormal trajectories \cite{max_sre},\cite{Agrachev_Exp_Map}. We now consider the normal extremal trajectories. The Hamiltonian (\ref{eq:4-2}) in this case can be written as:
\begin{equation}
h_{u}^{-1}(\lambda)=u_{1}h_{1}(\lambda)+u_{2}h_{2}(\lambda)-\frac{1}{2}\left(u_{1}^{2}+u_{2}^{2}\right),\qquad u\in\mathbb{R}^{2}.\label{eq:4-5}
\end{equation}
Using the maximization condition of PMP, the trajectories of the normal Hamiltonian satisfy the equalities:
\[
\frac{\partial h_{u}^{-1}}{\partial u}=\left(\begin{array}{c}
h_{1}-u_{1}\\
h_{2}-u_{2}
\end{array}\right)=0 \implies u_{1}=h_{1},\quad u_{2}=h_{2}.
\]

Normal extremals are the trajectories of Hamiltonian system $\dot{\lambda}=\overrightarrow{H}(\lambda),\,\lambda\in T^{*}M$,
with the maximized Hamiltonian $H=\frac{1}{2}\left(h_{1}^{2}+h_{2}^{2}\right)\geq 0$.
Specifically, for the non-constant normal extremals $H>0$. Note that the Hamiltonian function in the normal case is homogeneous w.r.t. $h_{1}$, $h_{2}$ and therefore we consider its trajectories for the level surface \textbf{$H=\frac{1}{2}$}. The phase cylinder containing the initial covector $\lambda$ in this case is:
\begin{equation}
C=T_{q_{0}}^{*}M\cap\left\{ H(\lambda)=\frac{1}{2}\right\} =\left\{ \left(h_{1},h_{2},h_{0}\right)\in\mathbb{R}^{3}\quad \vert \quad h_{1}^{2}+h_{2}^{2}=1\right\} .\label{eq:4-6}
\end{equation}
Differentiating $h_{i}$ w.r.t. the Hamiltonian vector field we get:
\begin{eqnarray*}
\dot{h}_{1} & = & \left\{ H,h_{1}\right\} =\left\{ \frac{1}{2}\left(h_{1}^{2}+h_{2}^{2}\right),h_{1}\right\} =h_{2}\left\{ h_{2},h_{1}\right\} =h_{2}h_{0},\\
\dot{h}_{2} & = & \left\{ H,h_{2}\right\} =\left\{ \frac{1}{2}\left(h_{1}^{2}+h_{2}^{2}\right),h_{2}\right\} =h_{1}\left\{ h_{1},h_{2}\right\} =-h_{1}h_{0},\\
\dot{h}_{0} & = & \left\{ H,h_{0}\right\} =\left\{ \frac{1}{2}\left(h_{1}^{2}+h_{2}^{2}\right),h_{0}\right\} =h_{1}\left\{ h_{1},h_{0}\right\} +h_{2}\left\{ h_{2},h_{0}\right\} =h_{1}h_{2}.
\end{eqnarray*}
Hence, the complete Hamiltonian system in the normal case is given as:
\begin{equation}
\left(\begin{array}{c}
\dot{h}_{1}\\
\dot{h}_{2}\\
\dot{h}_{0}\\
\dot{x}\\
\dot{y}\\
\dot{z}
\end{array}\right)=\left(\begin{array}{c}
h_{2}h_{0}\\
-h_{1}h_{0}\\
h_{1}h_{2}\\
h_{1}\cosh z\\
h_{1}\sinh z\\
h_{2}
\end{array}\right).\label{eq:4-7}
\end{equation}
Introduce the following coordinate:
\begin{equation}
h_{1}=\cos\alpha,\quad h_{2}=\sin\alpha.\label{eq:4-8}
\end{equation}
Then, the vertical subsystem is given as:
\begin{equation*}
\dot{\alpha}  = -h_{0}, \qquad \dot{h}_{0}=\frac{1}{2}\sin2\alpha .
\end{equation*}
Let us introduce another change of coordinates:
\begin{equation}
\gamma=2\alpha\in2S^{1}=\mathbb{R}/4\pi\mathbb{Z},\quad c=-2h_{0}\in\mathbb{R}.\label{eq:4-9}
\end{equation}
Then,
\begin{equation}
\left(\begin{array}{c}
\dot{\gamma}\\
\dot{c}
\end{array}\right)=\left(\begin{array}{c}
c\\
-\sin\gamma
\end{array}\right),\label{eq:4-10}
\end{equation}
which represents the double covering of a mathematical pendulum. Hence the vertical subsystem of the normal Hamiltonian system (\ref{eq:4-7}) is a standard pendulum.

\section{Parametrization of Extremal Trajectories}

\subsection{Hamiltonian System}

The Hamiltonian system for the normal trajectories was given in (\ref{eq:4-7}). Under the transformations introduced in (\ref{eq:4-8})-(\ref{eq:4-9}), the horizontal subsystem can be written as:

\begin{equation}
\left(\begin{array}{c}
\dot{x}\\
\dot{y}\\
\dot{z}
\end{array}\right)=\left(\begin{array}{c}
h_{1}\cosh z\\
h_{1}\sinh z\\
h_{2}
\end{array}\right)=\left(\begin{array}{c}
\cos\frac{\gamma}{2}\cosh z\\
\cos\frac{\gamma}{2}\sinh z\\
\sin\frac{\gamma}{2}
\end{array}\right).\label{eq:5-1}
\end{equation}

\subsection{Decomposition of the Initial Cylinder}

Following the techniques employed in \cite{max_sre}, the decomposition
of phase cylinder $C$ proceeds as follows. The total energy integral $E$
 of the pendulum obtained in (\ref{eq:4-10}) is given as:
\begin{equation*}
E=\frac{c^{2}}{2}-\cos\gamma=2h_{0}^{2}-h_{1}^{2}+h_{2}^{2},\qquad E\in[-1,+\infty).
\end{equation*}
The total energy $E$ of the pendulum is a constant of motion for the Hamiltonian vector field $\overrightarrow{H}$. The initial cylinder (\ref{eq:4-6}) may be decomposed into the following subsets based upon the pendulum energy that correspond to various pendulum trajectories:
\begin{eqnarray*}
C & = & \bigcup_{i=1}^{5}C_{i},
\end{eqnarray*}
where,
\begin{eqnarray*}
C_{1} & = & \{\lambda\in C \, \vert \, E\in(-1,1)\},\\
C_{2} & = & \{\lambda\in C \, \vert \, E\in(1,\infty)\},\\
C_{3} & = & \{\lambda\in C \, \vert \, E=1,c\neq0\},\\
C_{4} & = & \{\lambda\in C \, \vert \, E=-1, \,c=0 \}=\{(\gamma,c)\in C \, \vert \, \gamma=2\pi n,c=0\}\},\quad n\in\mathbb{N},\\
C_{5} & = & \{\lambda\in C \, \vert \, E=1, \,c=0 \}=\{(\gamma,c)\in C \, \vert \, \gamma=2\pi n+\pi,c=0\}\},\quad n\in\mathbb{N}.
\end{eqnarray*}
Continuing the approach taken in \cite{max_sre} the subsets $C_{i}$ may be further decomposed as:
\begin{eqnarray*}
C_{1} & = & \cup_{i=0}^{1}C_{1}^{i},\quad C_{1}^{i}=\{(\gamma,c)\in C_{1} \, \vert \, \mathrm{sgn}(\cos(\gamma/2))=(-1)^{i}\},\\
C_{2} & = & C_{2}^{+}\cup C_{2}^{-},\quad C_{2}^{\pm}=\{(\gamma,c)\in C_{2} \, \vert \, \mathrm{sgn}\, c=\pm1\},\\
C_{3} & = & \cup_{i=0}^{1}(C_{3}^{i+}\cup C_{3}^{i-}),\quad C_{3}^{i\pm}=\{(\gamma,c)\in C_{3} \, \vert \, \mathrm{sgn}(\cos(\gamma/2))=(-1)^{i},\mathrm{sgn}\, c=\pm1\},\\
C_{4} & = & \cup_{i=0}^{1}C_{4}^{i},\quad C_{4}^{i}=\{(\gamma,c)\in C \, \vert \, \gamma=2\pi i,c=0\},\\
C_{5} & = & \cup_{i=0}^{1}C_{5}^{i},\quad C_{5}^{i}=\{(\gamma,c)\in C \,\vert \, \gamma=2\pi i+\pi,c=0\}.
\end{eqnarray*}
In all of the above $i=0,1$. The decomposition of the initial cylinder $C$ is depicted in Figure \ref{fig:1}.

\begin{figure}
\begin{centering}
\includegraphics[scale=0.4]{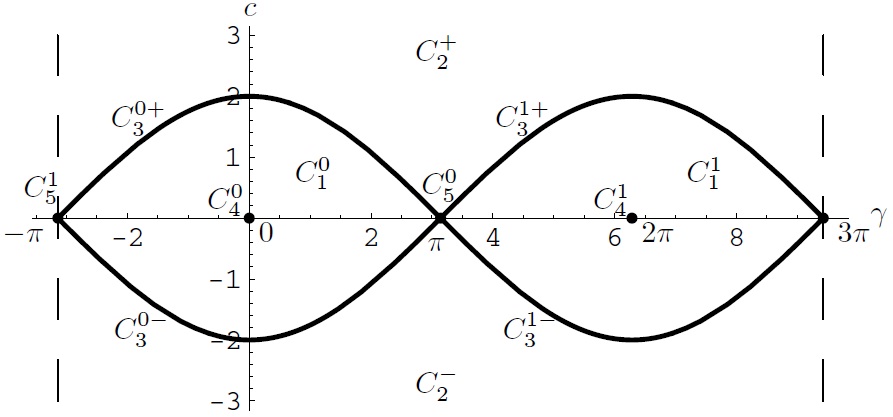}
\par\end{centering}

\caption{\label{fig:1}Decomposition of the Initial Cylinder and the Connected
Subsets}
\end{figure}


\subsection{Elliptic Coordinates}

Employing the approach developed in \cite{max_sre},\cite{Ardentov} we transform the system in terms of elliptic coordinates $(\varphi,k)$ on the domain $\cup_{i=1}^{3}C_{i}\subset C$. Note that $\varphi$ is the reparametrized time of motion and $k$ is the reparametrized energy of the pendulum. Correspondingly, we describe the system and the extremal trajectories in terms of the Jacobi elliptic functions $\textrm{sn}(\varphi,k)$, $\textrm{cn}(\varphi,k)$, $\textrm{dn}(\varphi,k)$, $\textrm{am}(\varphi,k)$, and $\textrm{E}(\varphi,k)=\intop_{0}^{\varphi}\textrm{dn}^{2}(t,k)dt$. Detailed description of the Jacobi elliptic functions may be found in \cite{Whittaker_Jacobi}.

\subsubsection{Case 1 : $\lambda=(\varphi,k)\in C_{1}$}

\begin{align}
k & =\sqrt{\frac{E+1}{2}}=\sqrt{\sin^{2}\frac{\gamma}{2}+\frac{c^{2}}{4}}\in(0,1),\label{eq:5-2}\\
\sin\frac{\gamma}{2} & =s_{1}k\,\textrm{sn}(\varphi,k),\quad s_{1}=sgn\left(\cos\frac{\gamma}{2}\right),\label{eq:5-3}\\
\cos\frac{\gamma}{2} & =s_{1}\textrm{dn}(\varphi,k),\label{eq:5-4}\\
\frac{c}{2} & =k\,\textrm{cn}(\varphi,k),\quad\varphi\in[0,4K(k)].\label{eq:5-5}
\end{align}

\begin{proposition}
In elliptic coordinates, the flow of the vertical subsystem rectifies. \end{proposition}
\begin{proof}
Using (\ref{eq:5-2}), we get $\dot{k}=0$ since $\dot{E}=0$.
Using (\ref{eq:5-5}), the derivatives of elliptic functions defined in \cite{Whittaker_Jacobi} and taking into account that $\dot{k}=0$, 
\begin{align*}
\frac{\dot{c}}{2} & =k\frac{d}{d\varphi}\textrm{cn}(\varphi,k).\frac{d\varphi}{dt},\\
-\sin\gamma & =-2k \, \textrm{sn}(\varphi,k)\textrm{dn}(\varphi,k)\dot{\varphi},
\end{align*}
thus we have,
\begin{equation}
\dot{\varphi}=\frac{\sin\gamma}{2k\,\textrm{sn}(\varphi,k).\textrm{dn}(\varphi,k)}.\label{eq:5-6}
\end{equation}
Now using (\ref{eq:5-3}),(\ref{eq:5-4}):
\begin{align*}
\sin\frac{\gamma}{2}\cos\frac{\gamma}{2} & =s_{1}k\,\textrm{sn}(\varphi,k).s_{1}\textrm{dn}(\varphi,k),\\
 \sin\gamma & =2k\,\textrm{sn}(\varphi,k).\textrm{dn}(\varphi,k).
\end{align*}
Thus (\ref{eq:5-6}) becomes:
\begin{equation*}
\dot{\varphi}=1.
\end{equation*}

\hfill$\square$
\end{proof}

\subsubsection{Case 2 : $\lambda=(\varphi,k)\in C_{2}$}

\begin{align}
k & =\sqrt{\frac{2}{E+1}}=\sqrt{\frac{1}{\sin^{2}\frac{\gamma}{2}+\frac{c^{2}}{4}}}\in(0,1),\label{eq:5-7}\\
\sin\frac{\gamma}{2} & =s_{2}\textrm{sn}\left(\frac{\varphi}{k},k\right),\quad s_{2}=\mathrm{sgn}(c),\label{eq:5-8}\\
\cos\frac{\gamma}{2} & =\textrm{cn}\left(\frac{\varphi}{k},k\right),\label{eq:5-9}\\
\frac{c}{2} & =\frac{s_{2}}{k}\textrm{dn}\left(\frac{\varphi}{k},k\right),\quad\varphi\in\left[0,4kK(k)\right].\label{eq:5-10}
\end{align}

\subsubsection{Case 3 : $\lambda=(\varphi,k)\in C_{3}$ }

\begin{align}
k & =1,\label{eq:5-11}\\
\sin\frac{\gamma}{2} & =s_{1}s_{2}\tanh\varphi,\quad s_{1}=\mathrm{sgn}\left(\cos\frac{\gamma}{2}\right),\quad s_{2}=\mathrm{sgn}(c),\label{eq:5-12}\\
\cos\frac{\gamma}{2} & =s_{1}/\cosh\varphi,\label{eq:5-13}\\
\frac{c}{2} & =s_{2}/\cosh\varphi,\quad\varphi\in(-\infty,\infty).\label{eq:5-14}
\end{align}
Using the procedure outlined for the Case 1, it can be proved that in the coordinates $(\varphi,k)$, the flow of the pendulum rectifies for Cases 2 and 3 as well. 

\subsection{Integration of the Vertical Subsystem}

Since the flow of vertical subsystem rectifies in the elliptic coordinates, therefore, the vertical subsystem is trivially integrated as $\varphi_{t}=t+\varphi$ and $k=constant$, where $\varphi$ is the value of $\varphi_{t}$ at $t=0$. 

\subsection{Integration of the Horizontal Subsystem}

In the following we consider integration of the horizontal subsystem (\ref{eq:5-1}) for Cases 1-3 noted above. Assuming zero initial state, i.e., $x(0)=y(0)=z(0)=0$ since $q_{0}=Id$.

\subsubsection{Case 1 : $\lambda=(\varphi,k)\in C_{1}$}
\begin{theorem}
Extremal trajectories in Case 1 are parametrized as follows:
\begin{equation}
\left(\begin{array}{c}
x_{t}\\
y_{t}\\
z_{t}
\end{array}\right)=\left(\begin{array}{c}
\frac{s_{1}}{2}\left[\left(w+\frac{1}{w\left(1-k^{2}\right)}\right)\left[\mathrm{E}(\varphi_{t})-\mathrm{E}(\varphi)\right]+\left(\frac{k}{w(1-k^{2})}-kw\right)\left[\mathrm{sn}\,\varphi_{t}-\mathrm{sn}\,\varphi \right]\right]\\
\frac{1}{2}\left[\left(w-\frac{1}{w\left(1-k^{2}\right)}\right)\left[\mathrm{E}(\varphi_{t})-\mathrm{E}(\varphi)\right]-\left(\frac{k}{w\left(1-k^{2}\right)}+kw\right)\left[\mathrm{sn}\,\varphi_{t}-\mathrm{sn}\, \varphi\right]\right]\\
s_{1}\ln\left[(\mathrm{dn}\,\varphi_{t}-k\mathrm{cn}\,\varphi_{t}).w\right]
\end{array}\right)\label{eq:5-15}
\end{equation}
where $w=\frac{1}{\mathrm{dn}\varphi-k\mathrm{cn}\varphi}$.\end{theorem}

\begin{proof}
From (\ref{eq:5-1}) consider $\dot{z}=\sin\frac{\gamma}{2}=s_{1}k\,\mathrm{sn}(\varphi,k)$.
The solution to this ODE can be written as:
\begin{equation*}
z_{t}=\intop_{\varphi}^{\varphi_{t}}s_{1}k\,\mathrm{sn}\,\varphi \, d\varphi.
\end{equation*}
Using the integration formulas for the Jacobi elliptic functions \cite{Table_Int}, we have:
\begin{equation*}
z_{t}=s_{1}\ln(\mathrm{dn}\,\varphi_{t}-k\mathrm{cn}\,\varphi_{t})-s_{1}\ln(\mathrm{dn}\,\varphi-k\mathrm{cn}\,\varphi).
\end{equation*}
Let $\ln w=-\ln(\mathrm{dn}\,\varphi-k\mathrm{cn}\,\varphi)$,
$w=\frac{1}{\mathrm{dn}\,\varphi-k\mathrm{cn}\,\varphi}$, then:
\begin{eqnarray*}
z_{t} & = & s_{1}\ln[(\mathrm{dn}\,\varphi_{t}-k\mathrm{cn}\,\varphi_{t}).w].
\end{eqnarray*}
From (\ref{eq:5-1}) now consider,
\begin{eqnarray*}
\dot{x} & = & \cos\frac{\gamma}{2}\cosh z=s_{1}\mathrm{dn}\,\varphi_{t}\cosh\left(s_{1}\ln\left[(\mathrm{dn}\,\varphi_{t}-k\mathrm{cn}\,\varphi_{t} ).w\right]\right),\\
\dot{x} & = & \frac{s_{1}}{2}\left(w.\mathrm{dn^{2}}\varphi_{t}-kw.\mathrm{dn}\,\varphi_{t}\mathrm{cn}\,\varphi_{t}+\frac{\mathrm{dn}\,\varphi_{t}}{(\mathrm{dn}\,\varphi_{t}-k\mathrm{cn}\,\varphi)_{t}.w}\right).\nonumber 
\end{eqnarray*}
This can be integrated as:
\begin{equation*}
x_{t}=\frac{s_{1}}{2}\left[w\intop_{\varphi}^{\varphi_{t}}\mathrm{\mathrm{dn^{2}}}\varphi_{t} d\varphi_{t}-kw\intop_{\varphi}^{\varphi_{t}}\mathrm{\mathrm{dn}}\,\varphi_{t}\mathrm{\mathrm{cn}}\,\varphi_{t} d\varphi_{t}+\frac{1}{w}\intop_{\varphi}^{\varphi_{t}}\frac{\mathrm{dn^{2}}\varphi_{t}+k\mathrm{cn}\,\varphi_{t}\mathrm{dn}\,\varphi_{t}}{\mathrm{dn^{2}}\varphi_{t}-k^{2}\mathrm{cn^{2}}\varphi_{t}}d\varphi_{t}\right].
\end{equation*}
Now using the standard identities of the elliptic functions, the result of integration can be written as:
\begin{equation*}
x_{t}=\frac{s_{1}}{2}\left[\left(w+\frac{1}{w\left(1-k^{2}\right)}\right)\left[\mathrm{E}(\varphi_{t})-\mathrm{E}(\varphi)\right]+\left(\frac{k}{w\left(1-k^{2}\right)}-kw\right)\left[\mathrm{sn}\,\varphi_{t}-\mathrm{sn}\,\varphi\right]\right].
\end{equation*}
From (\ref{eq:5-1}) now consider,
\begin{eqnarray*}
\dot{y} & = & \cos\frac{\gamma}{2}\sinh z=s_{1}\mathrm{dn}\,\varphi_{t}\sinh(s_{1}\ln[(\mathrm{dn}\,\varphi_{t}-k\mathrm{cn}\,\varphi)_{t}.w]),\nonumber \\
\dot{y} & = & \mathrm{dn}\,\varphi_{t}\sinh(\ln[(\mathrm{dn}\,\varphi_{t}-k\mathrm{cn}\,\varphi_{t}).w]).
\end{eqnarray*}
The integration follows the same pattern as described above and hence final result of integration of $\dot{y}$ can be written as:
\begin{equation*}
y_{t}=\frac{1}{2}\left[\left(w-\frac{1}{w\left(1-k^{2}\right)}\right)\left[\mathrm{E}(\varphi_{t})-\mathrm{E}(\varphi)\right]-\left(\frac{k}{w\left(1-k^{2}\right)}+kw\right)\left[\mathrm{sn}\,\varphi_{t}-\mathrm{sn}\,\varphi\right]\right].
\end{equation*} \hfill$\square$ 
\end{proof}

\subsubsection{Case 2 : $\lambda=(\varphi,k)\in C_{2}$ }
\begin{theorem}
\begin{flushleft}
Extremal trajectories in Case 2 are parametrized as follows:
\begin{eqnarray}
x_{t} & = & \frac{1}{2}\left(\frac{1}{w(1-k^{2})}-w\right)\left[\mathrm{E}(\psi_{t})-\mathrm{E}(\psi)-k^{\prime2}(\psi_{t}-\psi)\right]\nonumber \\
 & + & \frac{1}{2}\left(kw+\frac{k}{w(1-k^{2})}\right)\left[\mathrm{sn}\,\psi_{t}-\mathrm{sn}\,\psi\right],\nonumber \\
y_{t} & = & -\frac{s_{2}}{2}\left(\frac{1}{w(1-k^{2})}+w\right)\left[\mathrm{E}(\psi_{t})-\mathrm{E}(\psi)-k^{\prime2}(\psi_{t}-\psi)\right]\nonumber \\
 & + & \frac{s_{2}}{2}\left(kw-\frac{k}{w(1-k^{2})}\right)\left[\mathrm{sn}\,\psi_{t}-\mathrm{sn}\,\psi\right],\nonumber \\
z_{t} & = & s_{2}\ln[(\mathrm{dn}\,\psi_{t}-k\mathrm{cn}\,\psi_{t}).w],\label{eq:5-16}
\end{eqnarray}
where $w=\frac{1}{\mathrm{dn}\,\psi-k\mathrm{cn}\,\psi}$. 
\par\end{flushleft}\end{theorem}
\begin{proof}
Consider the horizontal system (\ref{eq:5-1}) for Case 2 (\ref{eq:5-7})-(\ref{eq:5-10})
and substitute $\psi=\frac{\varphi}{k}$ and $\psi_{t}=\frac{\varphi_{t}}{k}=\psi+\frac{t}{k}$. The proof of integration then follows from the procedure outlined in Case 1.\hfill$\square$ 
\end{proof}

\subsubsection{Case 3 : $\lambda=(\varphi,k)\in C_{3}$}
\begin{theorem}
Extremal trajectories in Case 3 are parametrized as follows:
\begin{equation}
\left(\begin{array}{c}
x_{t}\\
y_{t}\\
z_{t}
\end{array}\right)=\left(\begin{array}{c}
\frac{s_{1}}{2}\left[\frac{1}{w}\left(\varphi_{t}-\varphi \right)+w\left(\tanh\varphi_{t}-\tanh\varphi \right)\right]\\
\frac{s_{2}}{2}\left[\frac{1}{w}\left(\varphi_{t}-\varphi \right)-w\left(\tanh\varphi_{t}-\tanh\varphi \right)\right]\\
-s_{1}s_{2}\ln[w\, \mathrm{sech} \,\varphi_{t}]\label{eq:5-17}
\end{array}\right)
\end{equation}
where $w=\cosh\varphi$. \end{theorem}
\begin{proof}
Consider the horizontal system (\ref{eq:5-1}) for Case 3 (\ref{eq:5-11})-(\ref{eq:5-14}):
\begin{eqnarray*}
\dot{z} & = & \sin\frac{\gamma}{2}=s_{1}s_{2}\tanh\varphi,\\
z_{t} & = & -s_{1}s_{2}[\ln(\textrm{sech}\,\varphi_{t})-\ln(\textrm{sech}\,\varphi)].
\end{eqnarray*}
Let $-\ln(\mathrm{sech\,}\varphi)=\ln w$, $w=\cosh\varphi$,
then: 
\begin{equation*}
z_{t}=-s_{1}s_{2}\ln[w\,\textrm{sech}\,\varphi_{t}].
\end{equation*}
From (\ref{eq:5-1}) now consider,
\begin{eqnarray*}
\dot{x} & = & \cos\frac{\gamma}{2}\cosh z=s_{1}\textrm{sech}\,\varphi\cosh\left(-s_{1}s_{2}\ln[w\,\textrm{sech}\,\varphi]\right),\nonumber \\
\dot{x} & = & \frac{s_{1}\textrm{sech}\,\varphi}{2}\left[e^{\ln[w\,\textrm{sech}\,\varphi]}+e^{-\ln[w\,\textrm{sech}\,\varphi]}\right],\nonumber \\
\dot{x} & = & \frac{s_{1}\textrm{sech}\,\varphi}{2}\left[\frac{1+w^{2}\textrm{sech}^{2}\varphi}{w\,\textrm{sech}\,\varphi}\right],\nonumber \\
x_{t} & = & \frac{s_{1}}{2}\left[\frac{1}{w}\left(\varphi_{t}-\varphi \right)+w\left(\tanh\varphi_{t}-\tanh\varphi \right)\right].
\end{eqnarray*}
From (\ref{eq:5-1}) now consider,
\begin{eqnarray*}
\dot{y} & = & \cos\frac{\gamma}{2}\sinh z=s_{1}\textrm{sech}\,\varphi\sinh\left(-s_{1}s_{2}\ln[w\,\textrm{sech}\,\varphi]\right),\nonumber \\
\dot{y} & = & \frac{-s_{2}\,\textrm{sech}\,\varphi}{2}[e^{\ln[w\,\textrm{sech}\,\varphi]}-e^{-\ln[w\,\textrm{sech}\,\varphi]}],\nonumber \\
\dot{y} & = & \frac{-s_{2}\,\textrm{sech}\,\varphi}{2}\left[w\,\textrm{sech}\,\varphi-[w\,\textrm{sech}\,\varphi]^{-1}\right],\nonumber \\
y_{t} & = & \frac{s_{2}}{2}\left[\frac{1}{w}(\varphi_{t}-\varphi)-w(\tanh\varphi_{t}-\tanh\varphi)\right].
\end{eqnarray*} \hfill$\square$ 
\end{proof}
\subsection{Integration of the Horizontal Subsystem - The Degenerate Cases}

In the following we present integration of the horizontal subsystem in the degenerate cases, i.e., $\lambda\in C_{4}$ and $\lambda\in C_{5}$.

\subsubsection{Case 4 : $\lambda\in C_{4}$}
\begin{theorem}
Extremal trajectories in Case 4 are parametrized as follows:
\begin{equation}
\left(\begin{array}{c}
x_{t}\\
y_{t}\\
z_{t}
\end{array}\right)=\left(\begin{array}{c}
\mathrm{sgn}\left(\cos\frac{\gamma}{2}\right)t\\
0\\
0
\end{array}\right).\label{eq:5-18}
\end{equation}
\end{theorem}
\begin{proof}

\begin{eqnarray*}
\dot{z}& = & \sin\frac{\gamma}{2}=\sin\left(\frac{2n\pi}{2}\right)=0, \\ 
z_{t}& = & 0.
\end{eqnarray*}
Therefore,
\begin{eqnarray*}
\dot{x} & = & \cos\frac{\gamma}{2}\cosh z=\cos\left(\frac{2n\pi}{2}\right) = \mathrm{sgn}\left(\cos\frac{\gamma}{2}\right),\nonumber \\
x_{t} & = & \mathrm{sgn}\left(\cos\frac{\gamma}{2}\right)t.\nonumber \\
\end{eqnarray*}
Now,
\begin{eqnarray*}
\dot{y} & = & \cos\frac{\gamma}{2}\sinh z=\cos\left(\frac{2n\pi}{2}\right)\sinh(0)=0,\nonumber \\
y_{t} & = & 0.
\end{eqnarray*}
\hfill$\square$ 
\end{proof}

\subsubsection{Case 5 : $\lambda\in C_{5}$}
\begin{theorem}
Extremal trajectories in Case 5 are parametrized as follows:
\begin{equation}
\left(\begin{array}{c}
x_{t}\\
y_{t}\\
z_{t}
\end{array}\right)=\left(\begin{array}{c}
0\\
0\\
\mathrm{sgn}\left(\sin\frac{\gamma}{2}\right)t
\end{array}\right).\label{eq:5-19}
\end{equation}
\end{theorem}
\begin{proof}
\begin{eqnarray*}
\dot{z} & = & \sin\frac{\gamma}{2}=\sin\left(\frac{\pi+2n\pi}{2}\right)=\mathrm{sgn}\left(\sin\frac{\gamma}{2}\right), \\
z_{t} & = & \mathrm{sgn}\left(\sin\frac{\gamma}{2}\right)t.
\end{eqnarray*}
Now, 
\begin{eqnarray*}
\dot{x} & = & \cos\frac{\gamma}{2}\cosh z=\cos\left(\frac{\pi+2n\pi}{2}\right)\cosh z=0, \\
x_{t} & = & 0.
\end{eqnarray*}
Similarly,
\begin{eqnarray*}
\dot{y} & = & \cos\frac{\gamma}{2}\sinh z=\cos\left(\frac{\pi+2n\pi}{2}\right)\sinh z=0,\nonumber \\
y_{t} & = & 0.
\end{eqnarray*} \hfill$\square$ 
\end{proof}

\subsection{Qualitative Analysis of Projections of Extremal Trajectories on $xy$-Plane}

The standard formula for the curvature of a plane curve $(x(t),y(t))$ is given as \cite{Palais}:

\begin{equation}
\kappa=\frac{\dot{x}\ddot{y}-\ddot{x}\dot{y}}{\left(\dot{x}^{2}+\dot{y}^{2}\right)^{\frac{3}{2}}}.\label{eq:5-20}
\end{equation}
Using (\ref{eq:4-7}),(\ref{eq:5-20}) curvature of projections $\left(x(t),y(t)\right)$ of extremal trajectories of the Hamiltonian system (\ref{eq:4-7}) is given as: 
\begin{equation*}
\kappa=\frac{\sin\frac{\gamma}{2}}{\cos\frac{\gamma}{2}(\cosh2z)^{\frac{3}{2}}}.
\end{equation*}
The curves have inflection points when $\sin\frac{\gamma}{2}=0$ and cusps when $\cos\frac{\gamma}{2}=0$. We see that all the curves $\left(x(t),y(t)\right)$ have inflection points for $\lambda\in\cup_{i=1}^{3}C_{i}$ but only for $\lambda\in C_{2}$ the curves have cusps. The resulting trajectories $\left(x(t),y(t)\right)$ are shown in Figures \ref{fig:Inflectional-Trajectories C1}, \ref{fig:Inflectional-Trajectories C2}, \ref{fig:Inflectional-Trajectories C3}.
In degenerate Case 4, i.e., $\lambda\in C_{4}$, the extremal trajectories $q_{t}$ are Riemannian geodesics in the plane $\{z=0\}$. The curve $\left(x(t),y(t)\right)$ is a straight line on the $x$-axis. In Case 5, i.e., $\lambda\in C_{5}$, the curve $(x(t),y(t))$ is just the initial point $(0,0)$ for $\{x=y=0\}$. For non-zero initial conditions $x(0)=W_{x}$, $y(0)=W_{y}$, the motions of the pseudo Euclidean plane are only hyperbolic rotations whereas the translations along $x$-axis and $y$-axis are zero. The resulting trajectory is a quarter circle as $t\rightarrow \infty$ in the right sector of the unit hyperbola that is classically used to represent the pseudo Euclidean plane \cite{Minkowskian_SpaceTime}.  

\begin{figure}[ht]
\begin{centering}
\includegraphics[scale=0.4]{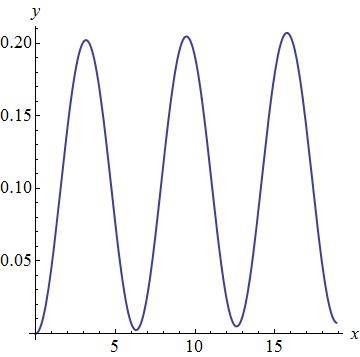}
\par\end{centering}

\centering{}\caption{\label{fig:Inflectional-Trajectories C1}Cuspless Trajectories,
$\lambda\in C_{1}$}
\end{figure}
\begin{figure}[ht]
\centering{}\includegraphics[scale=0.4]{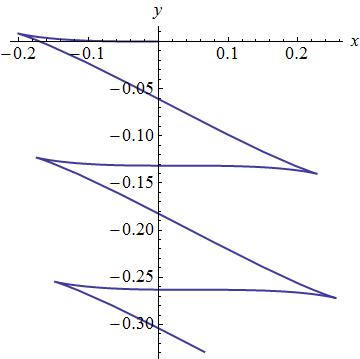}\caption{\label{fig:Inflectional-Trajectories C2}Trajectories with Cusps,
$\lambda\in C_{2}$}
\end{figure}
\begin{figure}[ht]
\centering{}\includegraphics[scale=0.4]{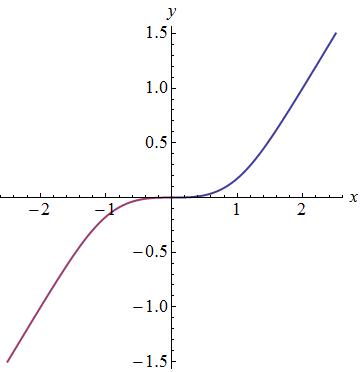}\caption{\label{fig:Inflectional-Trajectories C3}Critical Trajectories, 
$\lambda\in C_{3}$}
\end{figure}
\section{Discrete Symmetries and Maxwell Strata}
We now analyze symmetries in the vertical subsystem of the normal Hamiltonian system (\ref{eq:4-7}) to obtain a characterization of the Maxwell points. The analysis and organization of this section is based on the description of the Maxwell strata given
in \cite{max_sre}, \cite{Ardentov}, \cite{Sachkov_Dido_Symmetries},\cite{Euler_Elasticae_Sachkov} with corresponding results for the problem under consideration. 
\subsection{Symmetries of the Vertical Subsystem}

\subsubsection{Reflection Symmetries in the Vertical Subsystem}

Since the vertical subsystem of the Hamiltonian system is a mathematical pendulum (\ref{eq:4-10}), we exploit the reflection symmetries in the phase cylinder of the pendulum  to compute the discrete symmetries of the exponential mapping. The reflection symmetries in the phase portrait of a standard pendulum are given as:
\begin{equation}
\begin{alignedat}{1}\varepsilon^{1}:(\gamma,c) & \rightarrow(\gamma,-c),\\
\varepsilon^{2}:(\gamma,c) & \rightarrow(-\gamma,c),\\
\varepsilon^{3}:(\gamma,c) & \rightarrow(-\gamma,-c),\\
\varepsilon^{4}:(\gamma,c) & \rightarrow(\gamma+2\pi,c),\\
\varepsilon^{5}:(\gamma,c) & \rightarrow(\gamma+2\pi,-c),\\
\varepsilon^{6}:(\gamma,c) & \rightarrow(-\gamma+2\pi,c),\\
\varepsilon^{7}:(\gamma,c) & \rightarrow(-\gamma+2\pi,-c).
\end{alignedat}
\label{eq:6-1}
\end{equation}
Symmetries (\ref{eq:6-1}) form a symmetry group $G$ of parallelepiped with composition as group operation and $\varepsilon^{i}$ being the
elements of the group. The symmetries $\varepsilon^{3},\varepsilon^{4},\varepsilon^{7}$ preserve the direction of time, however, symmetries $\varepsilon^{1},\varepsilon^{2},\varepsilon^{5},\varepsilon^{6}$ reverse the direction of time \cite{max_sre}. As it is evident, symmetries where reflection about both axes of phase portrait occurs preserve the direction of time and others reverse the direction of time. 

\subsubsection{Reflections of Trajectories of the Pendulum}

Proposition 4.1 from \cite{max_sre} gives the transformations
that result in reflection of the phase portrait of pendulum and is reproduced here for sake of completeness.
\begin{proposition}\label{prop-6.1} Reflections (\ref{eq:6-1}) in the phase portrait of pendulum (\ref{eq:4-10}) are continued to the mappings \textit{$\varepsilon^{i}$ that transform trajectories $\delta_{s}=(\gamma_{s},c_{s})$ of the pendulum into the trajectories $\delta_{s}^{i}=(\gamma_{s}^{i},c_{s}^{i})$ as follows:}

\begin{equation}
\varepsilon^{i}:\delta=\{(\gamma_{s},c_{s})\vert s\in[0,t]\}\longmapsto\delta^{i}=\{(\gamma_{s}^{i},c_{s}^{i})\vert s\in[0,t]\},\quad i=1,\ldots,7,\label{eq:6-2}
\end{equation}
where, 
\begin{equation}
\begin{alignedat}{1}\delta^{1} & :(\gamma_{s}^{1},c_{s}^{1})=(\gamma_{t-s},-c_{t-s}),\\
\delta^{2} & :(\gamma_{s}^{2},c_{s}^{2})=(-\gamma_{t-s},c_{t-s}),\\
\delta^{3} & :(\gamma_{s}^{3},c_{s}^{3})=(-\gamma_{s},-c_{s}),\\
\delta^{4} & :(\gamma_{s}^{4},c_{s}^{4})=(\gamma_{s}+2\pi,c_{s}),\\
\delta^{5} & :(\gamma_{s}^{5},c_{s}^{5})=(\gamma_{t-s}+2\pi,-c_{t-s}),\\
\delta^{6} & :(\gamma_{s}^{6},c_{s}^{6})=(-\gamma_{t-s}+2\pi,c_{t-s}),\\
\delta^{7} & :(\gamma_{s}^{7},c_{s}^{7})=(-\gamma_{s}+2\pi,-c_{s}).
\end{alignedat}
\label{eq:6-3}
\end{equation}
\end{proposition}
\begin{figure}
\centering{}\includegraphics[scale=0.5]{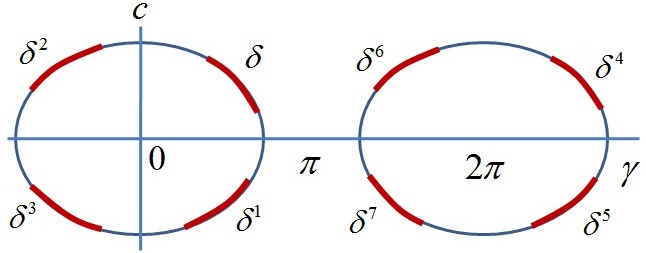}\caption{\label{fig3:Reflections}Reflections $\varepsilon^{i}:\delta\rightarrow\delta^{i}$
of trajectories of pendulum \cite{max_sre}}
\end{figure}
For the instant of time $s = t/2$, the reflections of extremal trajectories $\{(\gamma_s,c_s)\} \mapsto \{(\gamma_s^i,c_s^i)\}$ given by (\ref{eq:6-3}) reduce to the reflections of points $\{(\gamma,c)\} \mapsto \{(\gamma^i,c^i)\}$ given by (\ref{eq:6-1}). In this sense we write in Proposition 6.1 that the reflections are continued to the mappings $\varepsilon^{i}$. For proof see Proposition 4.1 {\cite{max_sre}}. Mappings (\ref{eq:6-2}) are shown in Figure \ref{fig3:Reflections}. 


\subsection{Symmetries of the Horizontal Subsystem}

\subsubsection{Reflections of Normal Extremals}

We now compute reflections of the normal extremals $q_s$ via the exponential mapping of the vertical subsystem. The canonical projection $\pi$ projects covectors from the cotangent bundle $T^{*}M$ to the manifold $M$, i.e., $\pi:\lambda \in T^{*}M\mapsto q \in M$.
The corresponding exponential map $Exp:N\rightarrow M$ of the arc-length parametrized normal extremal trajectories for $N=C \times \mathbb{R}^{+}$ is given as:
\begin{equation*}
Exp(\nu) =Exp(\lambda,s)=\pi \circ e^{s\overrightarrow{H}}(\lambda)=\pi(\lambda_{s})=q_{s},
\end{equation*}
where $\nu =(\lambda,s)\in N$ and $\lambda_{s}=(\gamma_{s},c_{s},q_{s})$ is the solution to the Hamiltonian system (\ref{eq:4-10}),(\ref{eq:5-1}). We analyze the reflections of the normal extremal trajectories of the horizontal subsystem corresponding to the reflection symmetries of the vertical subsystem. Action of the group $G$ on the normal extremals is defined as:

\begin{equation}
\varepsilon^{i}:\{\lambda_{s} \; \vert \; s \; \in[0,t]\}\mapsto\{\lambda_{s}^{i} \; \vert \; s \; \in[0,t]\},\qquad i=1,\ldots,7.\label{eq:6-4}
\end{equation}
The action $\varepsilon_{i}$ of the group $G$ on the vertical subsystem results in the reflection of trajectories of pendulum (\ref{eq:6-3}). The action of $G$ on the horizontal subsystem, i.e., the trajectories $q_{s}^{i}$ is described as follows:
\begin{proposition}
\label{prop:6.2} The image $q_{s}^{i}=(x{}_{s}^{i},y{}_{s}^{i},z{}_{s}^{i})$ of the normal extremal trajectory $q_{s}=(x_{s},y_{s},z_{s}),\; s\in[0,t]$ under the action of reflections $\varepsilon_{i}$ (\ref{eq:6-4}) is given as:

\begin{flalign*}
(1)\qquad z_{s}^{1} & =z_{t}-z_{t-s}, & {}\\
x_{s}^{1} & =\cosh z_{t}[x_{t}-x_{t-s}]-\sinh z_{t}[y_{t}-y_{t-s}], & {}\\
y_{s}^{1} & =\sinh z_{t}[x_{t}-x_{t-s}]-\cosh z_{t}[y_{t}-y_{t-s}]. & {}
\end{flalign*}
\begin{flalign*}
(2)\qquad z_{s}^{2} & =-[z_{t}-z_{t-s}], & {}\\
x_{s}^{2} & =\cosh z_{t}[x_{t}-x_{t-s}]-\sinh z_{t}[y_{t}-y_{t-s}], & {}\\
y_{s}^{2} & =-\sinh z_{t}[x_{t}-x_{t-s}]+\cosh z_{t}[y_{t}-y_{t-s}]. & {}
\end{flalign*}
\begin{flalign*}
(3)\qquad z_{s}^{3} & =-z_{s}, & {}\\
x_{s}^{3} & =x_{s}, & {}\\
y_{s}^{3} & =-y_{s}. & {}
\end{flalign*}
\begin{flalign*}
(4)\qquad z_{s}^{4} & =-z_{s}, & {}\\
x_{s}^{4} & =-x_{s}, & {}\\
y_{s}^{4} & =y_{s}. & {}
\end{flalign*}
\begin{flalign*}
(5)\qquad z_{s}^{5} & =-[z_{t}-z_{t-s}], & {}\\
x_{s}^{5} & =-\cosh z_{t}[x_{t}-x_{t-s}]+\sinh z_{t}[y_{t}-y_{t-s}], & {}\\
y_{s}^{5} & =\sinh z_{t}[x_{t}-x_{t-s}]-\cosh z_{t}[y_{t}-y_{t-s}]. & {}
\end{flalign*}
\begin{flalign*}
(6)\qquad z_{s}^{6} & =z_{t}-z_{t-s}, & {}\\
x_{s}^{6} & =-\cosh z_{t}[x_{t}-x_{t-s}]+\sinh z_{t}[y_{t}-y_{t-s}], & {}\\
y_{s}^{6} & =-\sinh z_{t}[x_{t}-x_{t-s}]+\cosh z_{t}[y_{t}-y_{t-s}]. & {}
\end{flalign*}
\begin{flalign*}
(7)\qquad z_{s}^{7} & =z_{s}, & {}\\
x_{s}^{7} & =-x_{s}, & {}\\
y_{s}^{7} & =-y_{s}. & {}
\end{flalign*}
\end{proposition}

\begin{proof}:
\textbf{Case 1 - Action of $\varepsilon^{1}:(\gamma_{s},c_{s},q_{s})\mapsto(\gamma_{s}^{1},c_{s}^{1},q_{s}^{1})=(\gamma_{t-s},-c_{t-s},q_{s}^{1})$
}
\begin{align}
\dot{z}_{s}^{1} & =\sin\frac{\gamma_{s}^{1}}{2},\nonumber \\
z_{s}^{1} & =\intop_{0}^{s}\sin\frac{\gamma_{r}^{1}}{2}dr\;=\intop_{0}^{s}\sin\frac{\gamma_{t-r}}{2}dr\;=-\left[\intop_{t}^{t-s}\sin\frac{\gamma_{p}}{2}dp\right] \;=z_{t}-z_{t-s}\label{eq:6-5}.
\end{align}
Similarly,

\begin{align}
\dot{x}_{s}^{1} & =\cos\frac{\gamma_{s}^{1}}{2}\cosh z_{s}^{1},\nonumber \\
x_{s}^{1} & =\intop_{0}^{s}\cos\frac{\gamma_{r}^{1}}{2}\cosh z_{r}^{1}dr\nonumber \\
& =\intop_{0}^{s}\cos\frac{\gamma_{t-r}}{2}\cosh(z_{t}-z_{t-r})dr\nonumber \\
& =-\intop_{t}^{t-s}\cos\frac{\gamma_{p}}{2}(\cosh z_{t}\cosh z_{p}-\sinh z_{t}\sinh z_{p})dp\nonumber \\
& =\cosh z_{t}[x_{t}-x_{t-s}]-\sinh z_{t}[y_{t}-y_{t-s}],
\end{align}
and
\begin{align}
\dot{y}_{s}^{1} & =\cos\frac{\gamma_{s}^{1}}{2}\sinh z_{s}^{1},\nonumber \\
y_{s}^{1} & =\intop_{0}^{s}\cos\frac{\gamma_{r}^{1}}{2}\sinh z_{r}^{1}dr\nonumber \\
& =\intop_{0}^{s}\cos\frac{\gamma_{t-r}}{2}\sinh(z_{t}-z_{t-r})dr\nonumber \\
& =-\intop_{t}^{t-s}\cos\frac{\gamma_{p}}{2}(\sinh z_{t}\cosh z_{p}-\cosh z_{t}\sinh z_{p})dp\nonumber \\
& =\sinh z_{t}[x_{t}-x_{t-s}]-\cosh z_{t}[y_{t}-y_{t-s}].
\end{align}
Proof of all other cases is similar.\hfill$\square$ 
\end{proof}

\subsubsection{Reflections of Endpoints of Extremal Trajectories}

Let us now consider the transformation of endpoints of  extremal trajectories resulting from action of the reflections $\varepsilon_{i}$ in the state space $M$:
\[
\varepsilon^{i}:q_{t}\rightarrow q_{t}^{i}. 
\] 
It can be readily seen from Proposition \ref{prop:6.2} that the point $q_{t}^{i}$ depends only on the endpoint $q_{t}$ and not on the whole trajectory $\{q_{s}\vert s\in[0,t]\}$. This is required to calculate the boundary conditions in the description of Maxwell strata corresponding to the reflection of extremal trajectories. 
\begin{proposition}
\label{prop:6.3} \textit{The action of reflections on endpoints of extremal trajectories can be defined as $\varepsilon^{i}:q\mapsto q^{i}$, where $q=(x,y,z)\in M,\quad q^{i}=(x^{i},y^{i},z^{i})\in M$ and,}
\end{proposition}
\begin{align}
(x^{1},y^{1},z^{1}) & =(x\cosh z-y\sinh z,\, x\sinh z-y\cosh z,\, z),\nonumber \\
(x^{2},y^{2},z^{2}) & =(x\cosh z-y\sinh z,\,-x\sinh z+y\cosh z,\,-z),\nonumber \\
(x^{3},y^{3},z^{3}) & =(x,\,-y,\,-z),\nonumber \\
(x^{4},y^{4},z^{4}) & =(-x,\, y,\,-z),\label{eq:6-6}\\
(x^{5},y^{5},z^{5}) & =(-x\cosh z+y\sinh z,\, x\sinh z-y\cosh z,\, -z),\nonumber \\
(x^{6},y^{6},z^{6}) & =(-x\cosh z+y\sinh z,\,-x\sinh z+y\cosh z,\, z),\nonumber \\
(x^{7},y^{7},z^{7}) & =(-x,\,-y,\, z).\nonumber 
\end{align}

\begin{proof}: 
Substitute $s=t$ and $(x_{0},y_{0},z_{0})=(0,0,0)$ in Proposition \ref{prop:6.2}.\hfill$\square$ 
\end{proof}
Notice that Proposition \ref{prop:6.3} defines the action of reflections in the image of the exponential mapping.

\subsection{Reflections as Symmetries of the Exponential Mapping}

Here we calculate explicit formulas for initial values of trajectories of the pendulum corresponding to the reflections. These will be useful in characterizing the fixed points of the reflections in the preimage of the exponential map. The action of reflection in the preimage of exponential mapping is defined as: 
\[
\varepsilon^{i}:N\rightarrow N,\qquad\varepsilon^{i}:\nu=(\lambda,t)=(\gamma,c,t)\mapsto\nu^{i}=(\lambda^{i},t)=(\gamma^{i},c^{i},t),
\]
where $(\gamma,c)$ are the trajectories of the pendulum with initial conditions $(\gamma_{0},c_{0})$ and $(\gamma^{i},c^{i})$ are the
reflections of the trajectories with initial conditions $(\gamma_{0}^{i},c_{0}^{i})$. The following proposition (a reproduction of Proposition 4.4 \cite{max_sre}) gives explicit formulas for $(\gamma^{i},c^{i})$. 
\begin{proposition}
\textbf{\label{prop-6.4}} \textit{Let $\nu=(\lambda,t)=(\gamma,c,t)\in N,\quad \nu^{i}=(\lambda^{i},t)=(\gamma^{i},c^{i},t)\in N$.
Then,}
\begin{equation}
\begin{alignedat}{1}(\gamma^{1},c^{1}) & =(\gamma_{t},-c_{t}),\\
(\gamma^{2},c^{2}) & =(-\gamma_{t},c_{t}),\\
(\gamma^{3},c^{3}) & =(-\gamma,-c),\\
(\gamma^{4},c^{4}) & =(\gamma+2\pi,c),\\
(\gamma^{5},c^{5}) & =(\gamma_{t}+2\pi,-c_{t}),\\
(\gamma^{6},c^{6}) & =(-\gamma_{t}+2\pi,c_{t}),\\
(\gamma^{7},c^{7}) & =(-\gamma+2\pi,-c).
\end{alignedat}
\label{eq:6-7}
\end{equation}
\end{proposition}
\begin{proof}
Substitute $s=0$ in Proposition \ref{prop-6.1}. \hfill$\square$ 
\end{proof}
Equations (\ref{eq:6-6}) give the explicit formulas for reflection of endpoints of the extremal trajectories in the image of exponential map, whereas, equations (\ref{eq:6-7}) give explicit formulas for the action of reflections $\varepsilon^{i}$ on the initial points of the extremals in the preimage of the exponential mapping. The actions in $M$ and $N$ are both induced by reflections $\varepsilon^{i}$ on extremals. Therefore it follows that the reflections  $\varepsilon^{i}$ for $i=1, \ldots ,7$, are symmetries of the exponential map.
\begin{proposition}
\textbf{\label{prop-6.5}} \textit{For any $\nu\in N$ and $i=1,\ldots,7$, we have }
$\varepsilon^{i} \circ Exp(\nu) = Exp \circ \varepsilon^{i}(\nu)$.
\end{proposition}

\section{Maxwell Strata Corresponding to the Reflections\label{sec:Maxwell-Strata}}

\subsection{Maxwell Points and Maxwell Sets}

The theory of Maxwell points is well known see e.g. \cite{max_sre}, \cite{Ardentov}, \cite{Sachkov_Dido_Symmetries},\cite{Sachkov-Dido-Max}. In optimal control theory they signify the points where the competing extremal trajectories with the same cost functional cross each other. S. Jacquet proved that for an analytic problem, a trajectory cannot be optimal after a Maxwell point \cite{Jacquet_Maxwell_Point}. Hence, they are important in the study of optimal trajectories. The set of all Maxwell points is called Maxwell Set. Let us consider a Maxwell set $MAX^{i},i=1,\ldots,7$ resulting from reflections $\varepsilon^{i}$:
\begin{equation}
MAX^{i}=\left\{ \text{\ensuremath{\nu}}=(\text{\ensuremath{\lambda}},t)\text{\ensuremath{\in}}N\quad|\quad\lambda\neq\lambda{}^{i},\quad Exp(\lambda,t)=Exp(\lambda^{i},t)\right\} .\label{eq:7-1}
\end{equation}
The corresponding Maxwell set in the image of the exponential mapping is defined as:
\[
Max^{i}=Exp(MAX^{i})\subset M.
\]
If $\nu=(\lambda,t) \in MAX^{i}$, then $q_t=Exp(\nu) \in Max^{i}$ is a Maxwell point along the trajectory $q_s=Exp(\lambda,s)$. Here we use the fact that if $\lambda \neq \lambda^i$, then $Exp(\lambda,s) \not\equiv Exp(\lambda^i,s)$.

\subsection{Fixed Points of Reflections in the Image of Exponential Map}

Since there are discrete symmetries of the exponential mapping, the idea is to exploit these symmetries and find the points where the
trajectories arising out of symmetries meet the normal extremal trajectory $q=(x,y,z)$. These points form the Maxwell set corresponding
to the reflection symmetries. Consider the following functions: 

\begin{equation}
R_{1}=y\cosh\frac{z}{2}-x\sinh\frac{z}{2},\qquad R_{2}=x\cosh\frac{z}{2}-y\sinh\frac{z}{2}.\label{eq:7-2}
\end{equation}
Consider $x,y$ in hyperbolic coordinates: 
\begin{equation*}
x=\rho\cosh\chi,\quad y=\rho\sinh\chi.
\end{equation*}
Thus $R_{1}$ and $R_{2}$ read as: 
\begin{align*}
R_{1} & =\rho\sinh\chi\cosh\frac{z}{2}-\rho\cosh\chi\sinh\frac{z}{2}=\rho\sinh(\chi-\frac{z}{2}),\\
R_{2} & =\rho\cosh\chi\cosh\frac{z}{2}-\rho\sinh\chi\sinh\frac{z}{2}=\rho\cosh(\chi-\frac{z}{2}).
\end{align*}

\begin{proposition}
\textbf{\label{prop-7.1}} Fixed points of the reflections $\varepsilon^{i}:q \mapsto q^{i}$ are given by the following conditions: 
\begin{flalign*}
(1)\qquad q^{1}=q\Longleftrightarrow & R_{1}(q)=0,\\
(2)\qquad q^{2}=q\Longleftrightarrow & z=0,\\
(3)\qquad q^{3}=q\Longleftrightarrow & y=0,\, z=0,\\
(4)\qquad q^{4}=q\Longleftrightarrow & x=0,\, z=0,\\
(5)\qquad q^{5}=q\Longleftrightarrow & x=y=z=0,\\
(6)\qquad q^{6}=q\Longleftrightarrow & R_{2}(q)=0,\\
(7)\qquad q^{7}=q\Longleftrightarrow & x=0,\, y=0. & {}
\end{flalign*}
\end{proposition}
\begin{proof}
We prove only Case (1): $q^{1}=q$. The proof for all other cases is similar. From (\ref{eq:6-6}), $x^{1}=x$ is equivalent to:
\begin{equation*}
x\cosh z-y\sinh z =x,\nonumber
\end{equation*}
which is equivalent to, 
\begin{align}
R_{1}\sinh\frac{z}{2} & =0.\label{eq:7-3}
\end{align}
Similarly, $y^{1}=y$ is equivalent to:
\begin{align}
R_{1}\cosh\frac{z}{2} & =0.\label{eq:7-4}
\end{align}

Eqs (\ref{eq:7-3}),(\ref{eq:7-4}) imply that $R_{1}=0.$ Hence case (1) of Proposition \ref{prop-7.1} is proved. \hfill$\square$

It can be observed readily that $Max^{i}$ for $i=3,4,5,7$ form 0 or 1 dimensional manifolds contained in 2-dimensional manifolds formed by $Max^{i}$
for $i=1,2,6$. Thus we consider only the 2-dimensional Maxwell sets. 
\end{proof}
\subsection{Fixed Points of the Reflections in the Preimage of the Exponential Map}

In the previous section we considered the fixed points of reflections in $M$ directly characterizing the Maxwell sets containing points where $q=q^{i}.$ We now consider the fixed points in the preimage of the exponential map, i.e., the solutions  to the equations $\lambda=\lambda^{i}$ for the proper characterization of the Maxwell points. We use the following coordinates in the phase cylinder of the pendulum for further analysis:
\begin{eqnarray}
\tau & = & \frac{1}{2}\left(\varphi_{t}+\varphi\right),\, p=\frac{t}{2}\textrm{ when }\nu=(\lambda,t)\in N_{1} \cup N_{3},\nonumber \\
\tau & = & \frac{1}{2k}\left(\varphi_{t}+\varphi\right),\, p=\frac{t}{2k}\textrm{ when }\nu=(\lambda,t)\in N_{2}.\label{eq:7-5}
\end{eqnarray}

\begin{proposition}
\textbf{\label{prop-7.2}} Fixed points of the reflections $\varepsilon^{i},\, i=1,2,6,$ in the preimage
of the exponential map are given as:
\begin{flalign*}
(1)\qquad\lambda^{1}=\lambda & \Leftrightarrow \begin{array}{cc}
\mathrm{cn}\tau=0, & \lambda\in C_{1},
\end{array}  & {}\\
(2)\qquad\lambda^{2}=\lambda & \Leftrightarrow\left\{ \begin{array}{cc}
\mathrm{sn}\tau=0, & \lambda\in C_{1}\cup C_{2}\\
\tau=0, & \lambda\in C_{3}
\end{array}\right\},  & {}\\
(3)\qquad\lambda^{6}=\lambda & \Leftrightarrow \begin{array}{cc}
\mathrm{cn}\tau=0, & \lambda\in C_{2}.
\end{array}  & {}
\end{flalign*}
\end{proposition}
\begin{proof}
\textbf{Case 1 - $\lambda^{1}=\lambda$}. It follows from Proposition \ref{prop-6.4} that if $\lambda\in C_{1}$, then $\lambda^{i}\in C_{1}$. Using Proposition
\ref{prop-6.4},
\begin{equation}
\lambda^{1}=\lambda\Longleftrightarrow\gamma_{t}=\gamma,\qquad-c_{t}=c.
\end{equation}
Using elliptic coordinates (\ref{eq:5-2})-(\ref{eq:5-5}) we have: 
\begin{equation}
\sin\frac{\gamma}{2}=s_{1}k\,\mathrm{sn}\varphi\implies\sin\frac{\gamma_{t}}{2}=s_{1}k\,\mathrm{sn}\varphi_{t}\implies\sin\frac{\gamma}{2}=s_{1}k\,\mathrm{sn}\varphi_{t}\implies\mathrm{sn}\varphi_{t}=\mathrm{sn}\varphi.\label{eq:7-6}
\end{equation}
\begin{equation}
\cos\frac{\gamma}{2}=s_{1}\mathrm{dn}\varphi\implies\cos\frac{\gamma_{t}}{2}=s_{1}\mathrm{dn}\varphi_{t}\implies\cos\frac{\gamma}{2}=s_{1}\mathrm{dn}\varphi_{t}\implies\mathrm{dn}\varphi_{t}=\mathrm{dn}\varphi.\label{eq:7-7}
\end{equation}
\begin{equation}
\frac{c}{2}=k\,\mathrm{cn}\varphi\implies\frac{c_{t}}{2}=k\,\mathrm{cn}\varphi_{t}\implies\frac{-c}{2}=k\,\mathrm{cn}\varphi_{t}\implies\mathrm{cn}\varphi_{t}=-\mathrm{cn}\varphi.\label{eq:7-8}
\end{equation}
Now from \cite{Table_Int},
\[
\mathrm{cn}\tau=\mathrm{cn}\frac{\varphi_{t}+\varphi}{2}=\pm\sqrt{\frac{\mathrm{cn}(\varphi_{t}+\varphi)+\mathrm{dn}(\varphi_{t}+\varphi)}{1+\mathrm{dn}(\varphi_{t}+\varphi)}},
\]
Consider $\mathrm{cn}(\varphi_{t}+\varphi)+\mathrm{dn}(\varphi_{t}+\varphi)$,
\[
\mathrm{cn}(\varphi_{t}+\varphi)+\mathrm{dn}(\varphi_{t}+\varphi)=\frac{\mathrm{cn}\varphi_{t}\mathrm{cn}\varphi-\mathrm{sn}\varphi_{t}\mathrm{sn}\varphi\mathrm{dn}\varphi_{t}\mathrm{dn}\varphi}{1-k^{2}\mathrm{sn}^{2}\varphi_{t}\mathrm{sn}^{2}\varphi}+\frac{\mathrm{dn}\varphi_{t}\mathrm{dn}\varphi+k^{2}\mathrm{sn}\varphi_{t}\mathrm{sn}\varphi\mathrm{cn}\varphi_{t}\mathrm{cn}\varphi}{1-k^{2}\mathrm{sn}^{2}\varphi_{t}\mathrm{sn}^{2}\varphi},
\]
Using (\ref{eq:7-6})-(\ref{eq:7-8}):
\begin{eqnarray}
\mathrm{cn}(\varphi_{t}+\varphi)+\mathrm{dn}(\varphi_{t}+\varphi) & = & \frac{-\mathrm{cn}^{2}\varphi-\mathrm{sn}^{2}\varphi\mathrm{dn}^{2}\varphi+\mathrm{dn}^{2}\varphi+k^{2}\mathrm{sn}^{2}\varphi\mathrm{cn}^{2}\varphi}{1-k^{2}\mathrm{sn}^{2}\varphi_{t}\mathrm{sn}^{2}\varphi},\nonumber \\
 & = & \frac{-\mathrm{cn}^{2}\varphi+\mathrm{dn}^{2}\varphi\left(1-\mathrm{sn}^{2}\varphi\right)+\left(1-\mathrm{dn}^{2}\varphi\right)\mathrm{cn}^{2}\varphi}{1-k^{2}\mathrm{sn}^{2}\varphi_{t}\mathrm{sn}^{2}\varphi},\nonumber \\
 & = & \frac{-\left(1-\mathrm{dn}^{2}\varphi\right)\mathrm{cn}^{2}\varphi+\left(1-\mathrm{dn}^{2}\varphi\right)\mathrm{cn}^{2}\varphi}{1-k^{2}\mathrm{sn}^{2}\varphi_{t}\mathrm{sn}^{2}\varphi},\nonumber \\
\implies\mathrm{cn}\tau & = & 0.
\end{eqnarray}
For $\lambda\in C_{2}^{\pm}$, we have $\lambda^{1}\in C_{2}^{\mp}$ because $c$ inverses sign. Thus $\lambda=\lambda^{1}$ is impossible. Similarly
if $\lambda\in C_{3}^{i\pm}$, we have $\lambda^{1}\in C_{3}^{i\mp},i=0,1$ because $c$ and $\gamma$ are both inverted in sign. Hence $\lambda=\lambda^{1}$ is impossible.

The proof for all other cases is similar.\hfill$\square$
\end{proof}

\subsection{General Description of Maxwell Strata Generated by Reflections}

Propositions \ref{prop-7.1} and \ref{prop-7.2} give the multiple points in the image and fixed points in the preimage of
the exponential map respectively. We now collate the results from these propositions to give general conditions under which points $q\in M$
form part of the Maxwell sets. 
\begin{proposition}
\textbf{\label{prop-7.3}} For $\nu=(\lambda,t)\in\cup_{i=1}^{3}N_{i}$ 
and $q=\left(x,y,z\right)=Exp\left(\nu\right)$,
\begin{flalign*}
(1)\qquad\nu\in MAX^{1} & \Leftrightarrow\left\{ \begin{array}{ccc}
R_{1}(q)=0, & \mathrm{cn}\tau\neq0 & for\,\lambda\in C_{1},\\
R_{1}(q)=0, &  & for\,\lambda\in C_{2}\cup C_{3}.
\end{array}\right\}  & {}\\
(2)\qquad\nu\in MAX^{2} & \Leftrightarrow\left\{ \begin{array}{ccc}
z=0, & \mathrm{sn}\tau \neq 0 & \lambda\in C_{1}\cup C_{2},\\
z=0, & \tau \neq 0  & \lambda\in C_{3}.
\end{array}\right\}  & {}\\
(3)\qquad\nu\in MAX^{6} & \Leftrightarrow\left\{ \begin{array}{ccc}
R_{2}(q)=0, & \mathrm{cn}\tau\neq0 & \lambda\in C_{2},\\
R_{2}(q)=0, &  & \lambda\in C_{1}\cup C_{3}.
\end{array}\right\}  & {}
\end{flalign*}
\end{proposition}
\begin{proof}
Apply Propositions \ref{prop-7.1} and \ref{prop-7.2}.
\end{proof}
\section{Future Work\label{sec:Future-Work}}

The most natural extension of this work is the complete description of Maxwell strata and computation of the conjugate and cut loci. To this end the methods developed in \cite{max_sre},\cite{cut_sre1}, \cite{Sachkov-Dido-Max}, \cite{Conj_Euler} appear most relevant and shall be employed. Complete description of Maxwell strata entails computation of roots of the functions $R_{i}(q)$ (\ref{eq:7-2}). This shall also give the global bound on the cut time for sub-Riemannian problem on $\mathrm{SH(2)}$. Description of the global structure of exponential map and optimal synthesis is the ultimate goal to be addressed in the entire research on SH(2).

\section{Conclusion}

The group of motions of the pseudo Euclidean plane $\mathrm{SH(2)}$ as an abstract algebraic structure has its own significance and sub-Riemannian problem on $\mathrm{SH(2)}$ is important in the entire program of study of 3-dimensional Lie groups. In this paper we have obtained the complete parametrization of extremal trajectories in terms of the Jacobi elliptic functions and described the nature of projections of extremal trajectories on $xy$-plane. We used reflection symmetries of the vertical and horizontal subsystem to obtain  the general description of Maxwell strata. The sub-Riemannian problem on $\mathrm{SH(2)}$ and the corresponding results are analogous to the sub-Riemannian problem and the associated results on $\mathrm{SE(2)}$ \cite{max_sre}. The extremal trajectories in both problems are parametrized by same elliptic coordinates $(\varphi,k)$ and the computed Hamiltonian flow is given in terms of Jacobi elliptic functions. Similarly, in both problems the Maxwell sets $Max^{i}$ form a 2-dimensional manifold though for different reflection symmetries $\varepsilon^{i}$. Our ongoing work on the computation of bounds of the first Maxwell time and the first conjugate time in sub-Riemannian problem on $\mathrm{SH(2)}$ shall enable us to draw parallels with the corresponding results on $\mathrm{SE(2)}$ and allow us to explore any symmetry that might exist between the sub-Riemannian problem on $\mathrm{SH(2)}$ and on $\mathrm{SE(2)}$.


		\bibliography{ref}

\begin{thebibliography}{10}

\bibitem{Strichartz}
Robert~S. Strichartz.
\newblock Sub-{R}iemannian {G}eometry.
\newblock {\em Journal of Differential Geometry}, 24(2):221--263, 1986.

\bibitem{Montgomery_Book}
R.~Montgomery.
\newblock {\em A tour of sub-{R}iemannian geometries, their geodesics and
  applications}.
\newblock Number~91 in Mathematical Surveys and Monographs. American
  Mathematical Society, 2002.

\bibitem{Agrchev_Barilari_Boscain_SR}
A.~A. Agrachev, Davide Barilari, and Ugo Boscain.
\newblock {\em Introduction to {R}iemannian and Sub-{R}iemannian {G}eometry
  (from {H}amiltonian viewpoint)}.
\newblock Preprint {SISSA}, September 2012.

\bibitem{Gromov}
Mikhael Gromov.
\newblock {\em Carnot-{C}aratheodory spaces seen from within}, volume 144 of
  {\em Sub-Riemannian Geometry, Progress in Mathematics}.
\newblock Birkhäuser Basel, 1996.

\bibitem{Gershkovich}
A.~M. Vershik and V.~Ya. Gershkovich.
\newblock Nonholonomic dynamical systems, geometry of distributions ad
  variational problems.
\newblock {\em Dynamical Systems - 7, Itogi Nauki i Tekhniki. Ser. Sovrem.
  Probl. Mat. Fund. Napr., 16, VINITI}, pages 5--85, 1987.

\bibitem{R_W_Brocket}
R.~W. Brockett.
\newblock Control theory and singular {R}iemannian geometry.
\newblock {\em New Directions in Applied Mathematics}, pages 11--27, 1982.

\bibitem{SR-Examples}
Enrico~Le Donne.
\newblock Lecture notes on sub-{R}iemannian geometry.
\newblock {\em Preprint}, 2010.

\bibitem{Heisenberg_Group}
F.~Monroy and A.~Anzaldo-Meneses.
\newblock Optimal control on the {H}eisenberg group.
\newblock {\em Journal of Dynamical and Control System}, 5(4):473--499, 1999.

\bibitem{SO(3)}
Ugo Boscain and F.~Rossi.
\newblock Invariant {C}arnot-{C}aratheodory metrics on {$S^{3}$}, {$SO(3)$},
  {$SL(2)$} and {L}ens spaces.
\newblock {\em SIAM, Journal on Control and Optimization}, 47:1851--1878, 2008.

\bibitem{max_sre}
I.~Moiseev and Yuri~L. Sachkov.
\newblock Maxwell strata in sub-{R}iemannian problem on the group of motions of
  a plane.
\newblock {\em ESAIM: COCV}, 16:380--399, 2010.

\bibitem{cut_sre1}
Yuri~L. Sachkov.
\newblock Conjugate and cut time in the sub-{R}iemannian problem on the group
  of motions of a plane.
\newblock {\em ESAIM: COCV}, 16:1018--1039, 2010.

\bibitem{cut_sre2}
Yuri~L. Sachkov.
\newblock Cut locus and optimal synthesis in the sub-{R}iemannian problem on
  the group of motions of a plane.
\newblock {\em ESAIM: COCV}, 17:293--321, 2011.

\bibitem{Ardentov}
A.~A. Ardentov and Yu.~L. Sachkov.
\newblock Extremal trajectories in a nilpotent sub-{R}iemannian problem on the
  {E}ngel group.
\newblock {\em Sbornik: Mathematics}, 202(11):1593--1615, 2011.

\bibitem{Mazhitova}
A.~D. Mazhitova.
\newblock Sub-{R}iemannian geodesics on the three-dimensional solvable
  non-nilpotent {L}ie group {SOLV$^{-}$}.
\newblock {\em Journal of Dynamical and Control Systems}, pages 1--14, 2012.

\bibitem{Sachkov_Dido_Symmetries}
Yuri~L. Sachkov.
\newblock Discrete symmetries in the generalized {D}ido problem.
\newblock {\em Sbornik: Mathematics}, 197(2):235--257, 2006.

\bibitem{Sachkov-Dido-Max}
Yuri~L. Sachkov.
\newblock The {M}axwell set in the generalized {D}ido problem.
\newblock {\em Sbornik: Mathematics}, 197(4):595--621, 2006.

\bibitem{Wong}
Yung-Chow Wong.
\newblock Euclidean n-planes in pseudo-{E}uclidean spaces and differential
  geometry of {C}artan domains.
\newblock {\em Bulletin of the American Mathematical Society}, 75(2):409--414,
  1969.

\bibitem{Ja.Vilenkin}
N.~Ja. Vilenkin.
\newblock {\em Special Functions and Theory of Group Representations
  (Translations of Mathematical Monographs)}.
\newblock American Mathematical Society, revised edition, 1968.

\bibitem{Thurston}
W.~P. Thurston.
\newblock Three-dimensional manifolds, {K}leinian groups and hyperbolic
  geometry.
\newblock {\em Bulletin of American Mathematical Society (N.S.)},
  6(3):357--381, 1982.

\bibitem{agrachev_barilari}
Andrei Agrachev and Davide Barilari.
\newblock Sub-{R}iemannian structures on 3{D} {L}ie groups.
\newblock {\em Journal of Dynamical and Control Systems}, 18(1):21--44, 2012.

\bibitem{agrachev_sachkov}
A.~A. Agrachev and Yuri~L. Sachkov.
\newblock {\em Control {T}heory from the {G}eometric {V}iewpoint}.
\newblock Springer Verlag, 2004.

\bibitem{sachkov_lectures}
Yuri~L. Sachkov.
\newblock Control theory on {L}ie groups.
\newblock {\em Journal of Mathematical Sciences}, 156(3):381--439, 2009.

\bibitem{Montgomery_Sub-Riemannian}
R.~Montgomery.
\newblock Isoholonomic problems and some applications.
\newblock {\em Communication in Mathematical Physics}, 128:565--592, 1990.

\bibitem{Ravchevsky}
P.~K. Rashevsky.
\newblock About connecting two points of complete nonholonomic space by
  admissible curve.
\newblock {\em Uch Zapiski Ped}, pages 83--94, 1938.

\bibitem{Chow}
W.~L. Chow.
\newblock Uber systeme von linearen partiellen dierentialgleichungen erster
  ordnung.
\newblock {\em Mathematische Annalen}, 117:98--105, 1940.

\bibitem{Agrachev_Exp_Map}
A.~A. Agrachev.
\newblock Exponential mappings for contact sub-{R}iemannian structures.
\newblock {\em Journal Dynamical and Control Systems}, 2(3):321--358, 1996.

\bibitem{Whittaker_Jacobi}
E.~T. Whittaker and G.~N. Watson.
\newblock {\em A Course of Modern Analysis, An introduction to the general
  theory of infinite processes and of analytic functions; with an account of
  principal transcendental functions}.
\newblock Cambridge University Press, Cambridge, 1996.

\bibitem{Table_Int}
I.~S. Gradshteyn and I.~M. Ryzhik.
\newblock {\em Table of Integrals, Series, and Products}.
\newblock Academic Press, 7 edition, 2007.

\bibitem{Palais}
Richard~S. Palais.
\newblock {\em A Modern Course on Curves and Surfaces}.
\newblock Virtual Math Museum, 2003.

\bibitem{Minkowskian_SpaceTime}
Francesco Catoni, Dino Boccaletti, Roberto Cannata, Vincenzo Catoni, Enrico
  Nichelatti, and Paolo Zampetti.
\newblock {\em The Mathematics of {M}ikowskian Space-Time with an Introduction
  to Commutative Hypercomplex Numbers}.
\newblock Birkhauser Verlag AG, 2008.

\bibitem{Euler_Elasticae_Sachkov}
Yuri~L. Sachkov.
\newblock Maxwell strata in the {E}uler elastic problem.
\newblock {\em Journal of Dynamical and Control Systems}, 14(2):169--234, April
  2008.

\bibitem{Jacquet_Maxwell_Point}
S.~Jacquet.
\newblock Regularity of sub-{R}iemannian distance and cut locus.
\newblock {\em Univ. Stud. Firenze, Florence, Italy 1999.}, 35, May 1999.

\bibitem{Conj_Euler}
Yuri~L. Sachkov.
\newblock Conjugate points in the {E}uler elastic problem.
\newblock {\em Journal of Dynamical and Control Systems}, 14:409--439, July
  2008.

\end{thebibliography}
		\bibliographystyle{unsrt}

\end{document}